\documentclass{article}
\usepackage[utf8]{inputenc}
\usepackage{amsmath, amsfonts, amsthm, amssymb}
\usepackage{graphicx}
\usepackage[all]{xy}

\newtheorem{Pro}{Proposition}
\newtheorem{Teo}[Pro]{Theorem}
\newtheorem{Cor}[Pro]{Corollary}
\theoremstyle{definition}
\newtheorem{Def}[Pro]{Definition}
\newtheorem{Exp}[Pro]{Example}
\theoremstyle{remark}
\newtheorem{Obs}[Pro]{Remark}

\begin{document}

\title{MVW-rigs}
\author{Yuri A. Poveda \thanks{Universidad Tecnológica de Pereira.} \\ yapoveda@utp.edu.co \and Alejandro Estrada \thanks{Universidad Tecnológica de Pereira.} \\ alestrada@utp.edu.co}


\maketitle              

\begin{abstract}

In the context of the MV-algebras there is a class of them endowed with a product. A subclass of this class of MV-algebras has been studied by Dinola in \cite{DiNola2}; the MV-algebras product or MVP. Thus for example the MV-algebra $[0,1]$ is closed for the usual product between real numbers. It is known that this product respects the usual order that in turn coincides with the natural order associated with this MV-algebra. Similarly, the algebra of continuous functions of $[0,1]^n$ in $[0,1]$. In the present work the class of MV-algebras with products is characterized in a wider context than that presented by Dinola. From the properties of the universal algebra found in the MV-algebras of closed continuous functions for products it will be shown that this general context is more convenient to work properties analogous to commutative algebra. \\

As a result of this characterization a new algebraic structure is defined, which is an MV-algebra endowed with a product operation, which we will call MVW-rig (Weak-Rig Multivalued)  because of its close relation with the rigs defined in \cite{Menni}. This structure is defined with axioms of universal algebra, a good number of natural examples are presented in the MV-algebras environment and the first results concerning homomorphisms, ideals, quotients and subdirect products are established. In particular, his prime spectrum is studied which, with the co-zariski topology defined by Dubuc, Poveda in \cite{Yuri1} is compact. Consequently, a good number of results analogous to the theory of commutative rings and rigs are presented, with which this theory maintains a close relation. \\


\end{abstract}
\section{Introduction}
\begin{Def}

A MV-algebra is a set $(A,\oplus,\neg,0)$ with a closed binary operation $\oplus$ and a unary operation $\neg$ such that for every $x,y,z \in A$ the following equations are satisfied:
\begin{itemize}
\item []MV1) $x\oplus (y\oplus z)=(x\oplus y)\oplus z$
\item []MV2)\, $x\oplus y=y\oplus x$
\item []MV3)\, $ x\oplus 0=x$
\item []MV4)\, $ x\oplus \neg 0=\neg 0 $
\item []MV5)\, $ \neg \neg x=x  $
\item []MV6)\, $\neg (\neg x \oplus y) \oplus y = \neg (\neg y \oplus x) \oplus x$
\end{itemize}
\end{Def}

Given an MV-algebra $A$, we define the constant $u$ and two operations $\odot$ and $\ominus$ as follows:
\begin{flalign}
& u =_{def} \neg 0 &\\ 
& x \odot y =_{def} \neg ( \neg x \oplus \neg y)   & \\
& x \ominus y =_{def} \neg ( \neg x \oplus y) &  
\end{flalign}

It is known that all MV-algebra $A$ with the order $x \leq y$ if and only if $x \ominus y =0$, is a reticle with maximum and minimum.

Next we will recall the definition of homomorphism, ideal and prime ideal in an arbitrary MV-algebra.

\begin{Def}
Given two MV-algebras $A$ and $B$, a function  $f: A \to B$ is a homomorphism of MV-algebras if for all $x,y$ in $A$
\begin{itemize}
\item [i)]  $f(0)=0$
\item [ii)]  $ f(x \oplus y)=f(x) \oplus f(y)$ 
\item [iii)] $ f(\neg x) =\neg f(x)$
\end{itemize}  
\end{Def}

\begin{Def}
Given $ A $ a MV-algebra. A subset $ I $ of $ A $ is called MV-ideal of $ A $ if it satisfies the following properties:
\begin{itemize}
\item [i)] $0 \in I$
\item [ii)]  If $a \leq b$ and $b \in I$, then $a \in I$
\item [iii)]  If $a,b \in I$, then $a \oplus b \in I$
\end{itemize}
\end{Def}

\begin{Def}[MV-prime ideal]

A subset $ P $ of a MV-algebra $ A $ is a MV-prime ideal if it is an ideal of MV-algebra, and given  $a,b \in A$, $a \wedge b \in P$ implies $a \in P$ or $b \in P$
\end{Def}

\begin{Teo}[Chang's Representation Theorem]\label{TeoChang}
Any non-trivial MV-algebra is a subdirect product of MV-chains.
\end{Teo}

The following result is typical of the MV-algebras, but as it is not found in the literature and we will use it in the present work, we present it together with its proof.

\begin{Pro}\label{dismenos}
	In a MV-algebra $A$ we have that $(x_1 \oplus x_2) \ominus (y_1 \oplus y_2) \leq (x_1 \ominus y_1) \oplus (x_2 \ominus y_2)$.
\end{Pro}

\begin{proof}
Of the definitions of MV-algebra and join we have that:
	$1=  \neg (x_1 \oplus  x_2) \oplus (x_1 \oplus x_2) = \neg (x_1 \oplus x_2) \oplus (x_1 \vee y_1) \oplus (x_2 \vee y_2) = \neg (x_1 \oplus x_2) \oplus (x_1 \ominus y_1) \oplus y_1 \oplus (x_2 \ominus y_2) \oplus y_2 = \neg ((x_1 \oplus x_2) \ominus (y_1 \oplus y_2)) \oplus (x_1 \ominus y_1) \oplus (x_2 \ominus y_2)$ then $(x_1 \oplus x_2) \ominus (y_1 \oplus y_2) \leq (x_1 \ominus y_1) \oplus (x_2 \ominus y_2)$ by the definition of order.\\
\end{proof}
\textbf{Notation:} We represent the addition as: $ \bigoplus_{i=1}^n x_i = x_1 \oplus x_2 \oplus \cdots \oplus x_n$.

If we generalize the last property we have that $$\bigoplus ^{n} _{i=1} x_i \ominus \bigoplus ^{n} _{i_1} y_i \leq \bigoplus ^{n} _{i=1} {(x_i \ominus y_i)}.$$

\section{MVW-rigs}

\begin{Def}[MVW-rig]\label{DefMVWrig}
A \textbf{MVW-rig}  is a structure type $( A, \oplus, \cdot, \neg ,0 )$   with three operations $\neg$, $\oplus$ and $\cdot $ defined in $A$ (by abbreviation $a\cdot b = ab$), such that it fulfills the following axioms for every $a,b,c \in A$:
\begin{itemize}
\item[$i)$] $(A, \oplus, \neg , 0 )$ is a MV-algebra.
\item[$ii)$] $ (A, \cdot)$ is an associative operation defined in $A$.
\item[$iii)$] $ a  0 = 0 a = 0$.
\item[$iv)$]  $a (b \oplus c) \leq a b \oplus a c $ and  $(b \oplus c) a \leq b a \oplus c a $.
\item[$v)$] $a (b \ominus c) \geq a b \ominus a c $ and  $(b \ominus c) a \geq b a \ominus c a $.
\end{itemize}
\end{Def}

By properties of the MV-algebra, $ 0 \leq a$ for all $a \in A$. We define $u =_{def} \neg 0$ for any MVW-rig. Then $ a \leq u$ for all $a \in A$. The operation $\neg$ is called \textbf{negation}, the operation $\oplus$ \textbf{sum} and the operation $\cdot $ \textbf{product} or \textbf{multiplication}.

\textbf{Notation}: $a \cdot a = a^2$. In general $a \cdot a \cdot (n$ $times) = a ^n$. 

\begin{Pro}\label{propiedades}
Given $A$ a MVW-rig, and $a,b,c \in A$. Then
\begin{itemize}
\item[$i)$]  $a \leq b$ implies $ac \leq bc$ and $ca \leq cb$
\item[$ii)$]$a(b \vee c ) \geq ab \vee ac$ and $(b \vee c)a \geq ba \vee ca$,
\item[$iii)$]$a(b \wedge c ) \leq ab \wedge ac$ and $(b \wedge c)a \leq ba \wedge ca$,
\item[$iv)$] $(a \vee b)^n \geq a^n \vee b^n $ with $n \in \mathbb{N}$.
\item[$v)$] $(a \wedge b)^n \leq a^n \wedge b^n $ with $n \in \mathbb{N}$.
\end{itemize}
\end{Pro}
\begin{proof}
$i)$ Given $a \leq b$ then $a \ominus b =0$, hence $ 0= (a\ominus b)c \geq ac \ominus bc  $ and $ 0= c (a \ominus b) \geq ca \ominus cb $ which implies that $ac \ominus bc = 0$ and $ca \ominus cb = 0$, and thus $ac \leq bc$ and $ca \leq cb$. $ii)$ It follows directly from the definition of join, from the fact that $b \vee c \geq b,c$; and by property $(i)$ $a(b \vee c) \geq ab, ac$. $iii)$ Similary to $ii)$. $iv)$ $(a \vee b)^n = (a \vee b)\cdots (a \vee b) \geq a^n, b^n$. Similarly we obtain $v)$
\end{proof}

If in a MVW-rig $ A $, there is an element $ s $ that has the property that for all $ x $ in $ A $ $ sx = xs = x $ it is said that $ s $ is a \textbf{unitary element} of the MVW-rig and $ A $ will be called \textbf {unitary MVW-rig}, in addition it is unique, since if there is an unitary element $ w \in A $ then $ s = sw = ws = w $. A unitary element in an MVW-rig will be denoted as 1. An MVW-rig is \textbf{commutative} if for all $x,y \in A$, $ xy = yx $.

\begin{Exp}\label{Ejemplousual}
The interval $ [0,1] $ of real numbers with the usual sum of MV-algebra $ [0,1] $ and usual multiplication in $\mathbb{R}$ is a commutative MVW-rig with unitary element where $u=1$ and $\neg x = 1-x$.
\end{Exp}
\begin{Exp}\label{Ejemplosemiusual}
	The interval $ [0, u] $ in $\mathbb{R}$ where $0 \leq u < 1$ is a non-unitary commutative MVW-rig with the truncated sum. If $u=0$ the structure is called \textbf{trivial MVW-rig}. 
\end{Exp}

\begin{Exp}\label{EjemploLukasiewicz}
	The set $\L_n = \left\{0,\frac{1}{n-1},\frac{2}{n-1},\cdots,\frac{n-2}{n-1}, 1\right\} $, known as MV-algebra of \L ukasiewicz is not a MVW-rig because it is not closed for the product. If we close it for products then we get the MVW-rig  $\dot{\L_n} =  \{ \frac{m}{n^k} \in \mathbb{Q} $ for every $k \in \mathbb{N}$ and all integer $m$ between $0$ and $n \} $.
\end{Exp}
\begin{Exp}
	Given $n \in \mathbb{N}$, the set $Z_n = \{ 0, 1, \dots , n \}$ is a MVW-rig with $u = n$ as a strong unit and operations defined as follows: $x \oplus y = $ min$\{n, x + y \} $, $\neg x = n-x$ and $x y = $ min$\{n, x \cdot y \} $ where operations sum $+$ and product $\cdot$ are the usual ones in the natural numbers and the order relation of the MV-algebra is the usual one in the natural ones. This MV-algebra is isomorphic to the MV-algebra $\L_{n+1}$ by applying $ \phi_n : \L_{n+1} \rightarrow Z_n$, $ \phi_n(x) = nx$, but $\L_{n+1} $ is not an MVW-rig. 
	
	This MVW-rig has some interesting properties: It has unitary element and is different from $ u $ if $ n> 1 $; has no cancellation property, the product between two elements is greater than or equal to them. This MVW-rig is a good source of counterexamples of properties that may be true for other MVW-rigs.
	
\end{Exp}
\begin{Exp}\label{EjemploFree}
	Given the MV-algebra $Free_1$, through the Mundici functor we get the $l_u$-group $Free_1^*$ which is isomorphic to the set of continuous functions of $[0,1]$ in $\mathbb{R}$ that have the property that each of them is constituted by finite linear polynomials with integer coefficients and that is contained in the $l_u$-ring $\mathcal{C}(\mathbb{R}^{[0,1]})$. Thus we can take the $l_u$-ring generated by $Free_1^*$ in $\mathcal{C}(\mathbb{R}^{[0,1]})$ which we will call $\dot{F}[x]$. This $l_u$-ring is isomorphic to $l_u$-ring of functions in $\mathcal{C}(\mathbb{R}^{[0,1]})$  each consisting of finite polynomials of $\mathbb{Z}[x]$. $\dot{F}[x]$ is a $l_u$-commutative ring. Given $u$ strong unit of $\dot{F}[x]$ such that $u^2 \leq u$ we take $\Gamma( \dot{F}[x],u) = \{f \in \dot{F}[x] \mid 0 \leq f \leq u \} $. This MV-algebra with the usual product of functions is a commutative MVW-rig denoted by $F_u[x]$. The MVW-rig $F_u[x]$ has more elements than the MV-algebra $ Free_1 $ since in $F_u[x]$ there are piecewise polynomial functions whose polynomials are of degree greater than 1 and the intervals of definition have by ends algebraic numbers. However $Free_1 \subset F_u[x]$. See \cite{Steven}.
	
	Unlike the previous example, this MVW-rig satisfies $fg \leq f \wedge g$ for all $f,g  \in F_u[x]$.
	
\end{Exp}
\begin{Exp}\label{Matrices}
	\textit{MVW-rig of matrices:} Given $ M_n $ the set of square matrices $n \times n$ with entries in $[0,1/n]$. We define the sum of two matrices $A,B \in M_n$ with the componentwise operation in the MV-algebra $[0,1/n]$, that is to say $A \oplus B = C$ with  $ c_{ij} =  a_{ij} \oplus b_{ij} $ for every $i,j = 1, \dots , n.$ The negation is defined componentwise in the MV-algebra $[0,1/n]$, that is  $ \neg A = C$, where $c_{ij} = \neg a_{ij} $ for all $i,j = 1, \dots , n.$. The set $ M_n $ is a MV-algebra with the described operations and the zero matrix. The natural order in $ M_n $ is given by: $A \leq B \Leftrightarrow_{def} (A)_{ij} \leq (B)_{ij} $ for every $i,j = 1, \dots , n.$ This defines a partial order in $M_n$. We now define the product in $M_n$ as $(AB)_{ij} = \bigoplus_{k=1}^n a_{ik} b_{kj} $ where each product and sum is defined in the MVW-rig $[0,1/n]$. Note that the strong unit in this MVW-rig is the matrix

	\[ U= \begin{pmatrix}
	1/n & \dots & 1/n\\
	\vdots & \ddots & \vdots \\
	1/n & \dots & 1/n
	\end{pmatrix} \]
	
	Let’s prove each of the MVW-rig axioms: $i)$ We know that  $M_n$ is a MV-algebra. $ii)$ Given three matrices $A,B,C \in M_n$ we have $(A(BC))_{ij}$=$ ( A ( \bigoplus_{k=1}^n b_{ik} c_{kj} ))_{ij} = \bigoplus_{r=1}^n a_{ir} (\bigoplus_{k=1}^n b_{ik} c_{kj}) =  \bigoplus_{r=1}^n \bigoplus_{k=1}^n a_{ir} b_{ik} c_{kj} =   \bigoplus_{k=1}^n \bigoplus_{r=1}^n a_{ir} b_{ik} c_{kj} = \\  \bigoplus_{r=1}^n ( \bigoplus_{k=1}^n a_{ir} b_{ik})  c_{kj} =  \bigoplus_{r=1}^n (AB)_{ik} c_{kj} = ((AB)C)_{ij}$ where the equality in the distributive law is true when the sum doesn’t overpass $u$ and in this case its proved. $iii)$ It follows directly from the definitions. $iv)$ $(A(B \oplus C))_{ij} = \bigoplus_{k=1}^n a_{ik}(b_{kj} \oplus c_{kj}) \leq   \bigoplus_{k=1}^n (a_{ik}b_{kj} \oplus a_{ik}c_{kj}) =   \bigoplus_{k=1}^n a_{ik}b_{kj} \oplus \bigoplus_{k=1}^n a_{ik}c_{kj} = (AB)_{ij} \oplus (AC)_{ij}$. $v)$. Given $(A(B \ominus C))_{ij} = \bigoplus_{k=1}^n a_{ik}(b_{kj} \ominus c_{kj}) \geq   \bigoplus_{k=1}^n (a_{ik}b_{kj} \ominus a_{ik}c_{kj})$ and by proposition (\ref{dismenos}) $\bigoplus_{k=1}^n (a_{ik}b_{kj} \ominus a_{ik}c_{kj})  \geq \bigoplus_{k=1}^n a_{ik}b_{kj} \ominus \bigoplus_{k=1}^n a_{ik}c_{kj}\\ = (AB)_{ij} \ominus (AC)_{ij}$.
\end{Exp}

\begin{Obs}
	A $l_u$-ring $R$ is a lattice ordered ring such that $(R, +, u)$ is a lattice ordered group with strong unit.
\end{Obs}

\begin{Pro}\label{Ejemplogamma}
	Given $l_u$-ring $R$ for which $u^2\leq u$, we have $\Gamma(R,u)$ is a MVW-rig, where $$\Gamma(R,u)=\{ x \mid 0 \leq x \leq u \}$$ with the truncated sum and the product of $ R $.
\end{Pro}
\begin{proof}
$i)$ $\Gamma(R,u)$ is a MV-algebra \cite{Yuri3}. $ii)$ As $u^2 \leq u$ then $xy \leq u$ for every $x,y \in \Gamma(R,u)$, then $(x \cdot y ) \cdot z = (xy) \cdot z = (xy)z = x(yz) = x\cdot (yz) = x \cdot (y \cdot z) $ where $\cdot$ is the product in $\Gamma(R,u)$. $iii)$ it follows directly of the definition. $iv)$ $x \cdot (y \oplus z) = x(y \oplus z) = x( (y +z) \wedge u ) \leq x(y+z) \wedge xu \leq x(y+z) \wedge u = (xy + xz) \wedge u =xy \oplus xz = x \cdot y \oplus x \cdot z$. The distributive on the left is similar. $v)$ $x(y \ominus z) = x((y -z) \vee 0) = x(y - z) \vee x 0 = x(y-z) \vee 0 =(xy - xz) \vee 0 = xy \ominus xz$. $v)$ $x \cdot 0 =  u \wedge x0 =  u \wedge 0  = 0 =  u \wedge 0  =   u \wedge 0x  = 0 \cdot x$.
\end{proof}

\section{Homomorphisms and ideals in the MVW-rigs}
\begin{Def}
Given $A$ and $B$ MVW-rigs. A function $f: A \rightarrow B$ is a homomorphism of MVW-rigs if the following properties are true:
\begin{itemize}
\item[$i)$]  $f$ is a homomorphism of MV-algebras.
\item[$ii)$] $f(ab)=f(a)f(b)$.
\end{itemize}
\end{Def}
\begin{Exp}
	Given the MVW-rig $F_u[x]$ and a function $\hat{a}$ such that it evaluates each function of $F_u[x]$ in a point $a\in [0,1]$. That is, $\hat{a}: F_u[x] \to [0,1]$, $\hat{a}(f)=f(a)$. The funtion $\hat{a}$ is a homomorphism. Let's see: $\hat{a}(0)=0(a)=0$, where $0$ if the zero funtion of $F_u[x]$. We also have that $\hat{a}(\neg f) = (\neg f)(a)= \neg f(a)=\neg \hat{a}(f)$ for every $f \in F_u[x]$.  Let's take $f,g \in F_u[x]$, then $\hat{a}(f \oplus g)= (f\oplus g )(a)= f(a) \oplus g(a) = \hat{a}(f) \oplus \hat{a}(g)$ and this proves $i)$. Given $f,g \in F_u[x]$ we have that $\hat{a}(f g)= (f g )(a)= f(a) g(a) = \hat{a}(f) \hat{a}(g)$ and this proves $ii)$. 
	The evaluation homomorphism is of great importance for relating sets $F_u[x]$ with $A=[0,u]$.
\end{Exp}

\begin{Def}
An ideal of an MVW-rig $A$ is a subset $I$ of $A$ that fulfills the following properties:
\begin{itemize}
\item[$i)$] $I$ is an MV-ideal of $A$ as MV-álgebra.
\item[$ii)$] If $a \in I$ and $b \in A$ then $ab \in I$ and $ba \in I$.
\end{itemize}
\end{Def}
\begin{Exp}
	In $F_u[x]$ we can take ideals analogous to the MV-ideals of MV-algebra $Free_1$ which are ideal in the MVW-rig $F_u[x]$. Let $z$ be a fixed element of $[0,1]$.
	\begin{itemize}
		\item $ f \in I_z \iff  f(z)=0 $
		\item $ f \in I_{z^+} \iff \exists \epsilon > 0 \text{ such that 
		}  f(x)=0, \forall x \in [z, z + \epsilon ] $.  ($z \neq 1$)
		\item $  f \in I_{z^-} \iff \exists \epsilon > 0 \text{ such that  
		}  f(x)=0, \forall x \in [z - \epsilon , z ]$. ($z \neq 0$)
		\item $ f \in I_S \iff f(x)=0 \forall x \in S, S \subset [0,1] $.
	\end{itemize}
\end{Exp}

\begin{Pro}
If $S$ is a subset of a commutative MVW-rig $ A $, then the ideal generated by $ S $ in $ A $ is:
\[ \langle S \rangle = \{ x \in A \mid x \leq \bigoplus_{i=1}^n a_i s_i, s_i \in S, a_i \in A \text{ or } a_i \in \mathbb{N} \text{ for each } i=1, \dots , n \} \]
\textbf{Note:} If $a_i = n \in \mathbb{N}$, $ns= s \oplus \cdots \oplus s$ ($n$ times).
\end{Pro}
\begin{proof}
	$ \langle S \rangle $ is ideal because $0 \in \langle S \rangle $. Given $x , y \in \langle S \rangle $ then $x \leq \bigoplus_{i=1}^n a_i s_i$ and $y \leq \bigoplus_{j=1}^m a_j s_j$, then $ x \oplus y \leq \bigoplus_{i=1}^n a_i s_i \oplus  \bigoplus_{j=1}^m a_j s_j$ and so $x \oplus y \in \langle S \rangle $. Given $x \leq y \in \langle S \rangle$ then $x \leq y \leq \bigoplus_{i=1}^n a_i s_i$ and therefore $x \in \langle S \rangle $. Lastly, given $x \in \langle S \rangle$ and $z \in A $ then $ zx \leq \bigoplus_{i=1}^n (z a_i) s_i $ this implies that $zx \in \langle S \rangle $ and the same for the case $xz \in \langle S \rangle$. \\
	To see that $\langle S \rangle$ is the less ideal of $A$ that contains $S$, let's take an ideal $I$ from $A$ that contains $S$ and let's see that $\langle S \rangle \subset I$. Given $x \in \langle S \rangle$ then $x \leq \bigoplus_{i=1}^n a_i s_i$. Since $I$ contains $S$ then every $s_i$ belongs to  $I$ and therefore $a_i s_i $ belongs to $I$. Then the sum $\bigoplus_{i=1}^n a_i s_i$ belongs to $I$ and this implies that $x \in I$.
\end{proof}

\begin{Pro}\label{prohomo}
Given $\phi$ a homomorphism of a MVW-rigs $A$ to a MVW-rig $B$, we have the following properties:
\begin{itemize}
\item[$i)$] If $S$ is a subMVW-rig of $A$ then $\phi (S) $ is a subMVW-rig of $B$.
\item[$ii)$] $\phi(x) \leq \phi(y)$ if and only if $ x \ominus y \in $ Ker$(\phi)$.
\item[$iii)$]If  $J$ is ideal of $B$ then $\phi^{-1}(J)$ is ideal of $A$.
\item[$iv)$]If $A$ is unitary and $\phi(1) \neq 0$  then $\phi(1)$ is the unitary element of $\phi(A)$.
\item[$v)$]$\phi$ is inyective if and only if Ker$(\phi) = \{ 0 \} $.
\end{itemize}
\end{Pro}
\begin{proof}
The definition of homomorphism of MVW-rig is directly followed.
\end{proof}

We will use this property of MV-algebras in the following proposition: $x \leq y \oplus z \iff x \ominus z \leq y$.

\begin{Pro}
	Given $I$ an ideal in a MVW-rig $A$ the equivalence relation. $x \equiv _I y \Leftrightarrow_{def} (x \ominus y) \oplus (y \ominus x) \in I$ is a congruence in the category of MVW-rigs.
	
\end{Pro}
\begin{proof}
It is known that this relation respects addition and negation. It is enough to show that it respects the product. If $a \equiv b \ mod(I)$ and $c \equiv d \ mod(I)$ then $a \ominus b \in I$ and $c \ominus d \in I$. We have that $ac \leq (a \vee b)(c \vee d) = ((a \ominus b) \oplus b)((c \ominus d) \oplus d)$ by definition of join; using distributive law of the axiom \ref {DefMVWrig}(iv) we have $ac \leq ((a \ominus b) \oplus b)((c \ominus d) \oplus d) \leq (a \ominus b) ((c \ominus d) \oplus d) \oplus b((c \ominus d) \oplus d) \leq (a \ominus b)(c \ominus d) \oplus (a \ominus b)d \oplus b(c \ominus d) \oplus bd$. So, $ac \ominus bd \leq (a \ominus b)(c \ominus d) \oplus (a \ominus b)d \oplus b(c \ominus d) \in I$ by absorbing property of $I$ regarding the product. Similarly $bd \ominus ac \in I$, then $(ac \ominus bd) \oplus (bd \ominus ac) \in I$ which implies $ac \equiv bd \ mod(I)$.
\end{proof}
\begin{Pro}
There is a bijection between the ideals of an MVW-rig $ A $ and the congruences in $A$.
\end{Pro}
\begin{proof}
Given  $ \equiv $ a congruence in $A$, the set $ I= \{ x \in A \mid x \equiv 0 \} $ is an ideal of MVW-rig. Just see that it has the absorbing property. If $x \equiv 0$ and $z \in A$ then since $\equiv $ preserves the product we have that $ xz \equiv 0z=0$. On the other hand, it was seen that $ \equiv_I $ is a congruence. The above assignment is bijective. Given $I,J$ ideals of $A$, $x \equiv_I y \iff x \equiv_J y \iff I=J.$ On the other hand, given $\equiv $ a congruence in $A$, $a \equiv b \iff a\ominus b \oplus b \ominus a \equiv 0 \iff a \equiv_I b$ with $I= \{ x \in A \mid x \equiv 0 \}$.    
\end{proof}

\section{Quotient MVW-rig}
We define the quotient $ A / I $ as the set of equivalence classes of $ x $ for each $x \in A$ which are denoted by $ [x] _I $. The set $ A / I $ has the operations:
\begin{flalign}
&  \neg [x]_I = [\neg x ]_I \label{negCoc} &\\  
& [x]_I \oplus [y]_I = [x \oplus y]_I  \label{sumCoc} & \\ 
& [x]_I [y]_I  = [x y]_I \label{ProdCoc} &  
\end{flalign}
The following proposition follows that $ \equiv_I $ is a congruence.
\begin{Pro}
$A/I$ is a MVW-rig. 
\end{Pro}

The correspondence $x \mapsto [x]_I$ defines a surjective homomorphism $ h $ from MVW-rig $ A $ to quotient MVW-rig $ A / I $ called the natural homomorphism of $ A $ onto $ A / I $ with $ \ker (h) = I$.
\begin{Teo}
Given $\phi$ a homomorphism of a MVW-rig $A$ to a MVW-rig $B$ with $ \ker (\phi)= K$, then there is a canonical isomorphism between $\phi(A)$ and $A/K$.
\[ \xymatrix{A \ar[rrrr]^{\phi} \ar[rrd]_{h} & & & & \phi(A) \subset B    \\  & & A/K  \ar@{^(->}[rru]_{\varphi}^{\cong} } \]
\end{Teo}
\begin{proof}
By the proposition \ref{prohomo}($i$) $\phi(A)$ is a MVW-rig. We define $\varphi : A/K \rightarrow \phi(A)$ by $\varphi([a]_K) = \phi(a)$. The Theorem 1.2.8 of \cite{Cignoli} shows that $\varphi$ is well defined, is one-to-one and onto in $\phi(A)$ with $\varphi([a]_K \oplus [b ]_K) = \varphi([a]_K) \oplus \varphi([b]_K)$. However, $\varphi([a]_K[b ]_K) = \varphi([ab]_K) = \phi(ab) = \phi(a) \phi(b)= \varphi([a]_K) \varphi([b]_K)$. So, $\varphi$ is an isomorphism of MVW-rigs.
\end{proof}

\begin{Teo}\label{Corresp_ideales}
Given $ A $ a MVW-rig, $ I $ an ideal of $ A $, then there exists a bijective correspondence between the ideals of $ A $ containing $ I $ and the ideals of the quotient MVW-rig $ A / I $ which preserves the relation of inclusion and also the direct and inverse image of an ideal is an ideal.
\end{Teo}
\begin{proof}
Let $f$ be the natural homomorphism of $A$ over $A/I$. Given $J$ ideal of $A$ that contains $I$ then Ker$(f) = I \subset J$, we want to see that $f(J)$ is ideal in $A/I$: $0 \in f(J)$ because Ker$(f) \subset J$, given $x,y \in f(J)$ and $z \in A/I$, then $x = f(a)$, $y = f(b)$, $z = f(k)$ with $a,b \in J$, $k \in A$; from here results $x \oplus y = f(a \oplus b)$, $zx = f(ka)$, $xz = f(ak)$, with which $x \oplus y,$ $xz$, $kx \in f(J)$, and if $z \leq x$ and $x \in f(J)$ then $k \ominus a \in Ker(f)$ by property (ii) of the proposition (\ref{prohomo}), then $k \ominus a \in J$ and since $a \in J$ then $(k \ominus a ) \oplus a = k \vee a \in J$ and this implies that $k \in J$, this is, $f(k)= z \in f(J)$. On the other hand, if $\tilde{J}$ is ideal of $A/I$, again by proposition (\ref{prohomo}) $f^{-1}(\tilde{J})$ is ideal of $A$ and contains $I$ because for $a \in I$, $[a]_I = [0]_I \in  \tilde{J}$. \\
Given $\mathcal{I}$ the collection of all the ideals of $A$ that contain $I$ and $\mathcal{I}_0$ the collection of all the ideals of $A/I$. \\
The correspondence $\widetilde{f} : \mathcal{I} \rightarrow \mathcal{I}_0$, $\widetilde{f}(J) := f(J) = \{ [a]_I \mid a \in J \} $ is a bijection. \\
$\widetilde{f}$ is injective since given $K,J \in \mathcal{I}$ with $\widetilde{f}(J) = \widetilde{f}(K)$, then $a \in J \Leftrightarrow [a]_I \in \widetilde{f}(J) \Leftrightarrow [a]_I \in \widetilde{f}(K) \Leftrightarrow a \in K$. (The implication $[a]_I \in \widetilde{f}(K) \Rightarrow a \in K$ is because if we have $b \in K$ given that $[a]_I = [b]_I$ then $a \ominus b  \in K$, then $(a \ominus b) \oplus b \in K$ and therefore $a \vee b \in K$ and so $a \in K$). \\
$\widetilde{f}$ is surjective because given $\tilde{J} \in \mathcal{I}_0$ we know that $f^{-1}(\tilde{J})$ is ideal of $A$ that contain $I$; and since $f$ is surjective $f(f^{-1}(\tilde{J}))= \tilde{J}$, that is to say, $\widetilde{f}(f^{-1}(\tilde{J}))= \tilde{J}$. \\
$\widetilde{f}$ preserves the inclusion because given $J \supseteq K$ in $\mathcal{I}$, if $[a]_I \in \widetilde{f}(K)$ then $a \in K$ and $a \in J$, then $[a]_I \in J$ which implies that $\widetilde{f}(J) \supseteq \widetilde{f}(K)$.
\end{proof}

\subsection{Prime and maximal ideals}

\begin{Def}
An ideal $I$ of a MVW-rig $A$ is prime if $ab \in I$ implies $a \in I$ or $b \in I$.
\end{Def}

\begin{Pro}\label{homoprimo}
Given $A,B$ MVW-rigs. If $f: A \rightarrow B$ is a homomorphism of MVW-rigs and $P$ is a prime ideal of $B$, then $f^{-1} (P) = \{ a \in A \mid f(a) \in P \}$ is a prime ideal of $A$.
\end{Pro}

In general, being an ideal MV-prime for $ A $ does not imply that it is a prime ideal for MVW-rig $ A $, it is not true either  otherwise . The implication is obtained only when we can establish a relation of order between the product and the infimum, as in the following proposition:

\begin{Pro}
Given a MVW-rig $A$ where $ab \leq a \wedge b$ for every $a,b \in A$. If $P$ is a prime ideal of $A$ then it is a MV-prime ideal.
\end{Pro}
\begin{proof}
Given $a,b \in A$ such that $a \wedge b \in P$, then $ab \in P$ by the relation $ab \leq a \wedge b$. Since $P$ is prime ideal of the MVW-rig $A$ then $a \in P $ or $b \in P$, and so $P$ is a MV-prime ideal.
\end{proof}

\begin{Def}
Given $M$ an ideal of a MVW-rig $A$. $M$ is a \textbf{maximal ideal} if for every $a \in A$ with $a \notin M$, $ \langle M, a  \rangle = A$.
\end{Def}

\begin{Pro}
In a commutative MVW-rig $A$ with unitary element, every maximal ideal is a prime ideal.
\end{Pro}
\begin{proof}
Given $M$ a maximal ideal and $a,b$ elements of $A$ such that $ab \in M$. Let's suppose that $a \notin M$, then there exists elements $m \in M$ and $x \in \langle a \rangle$ such that $1 \leq m \oplus x$. Then, $b = b1 \leq b(m \oplus x) \leq bm \oplus bx$, but $bm \in M$ and $bx \in \langle ab \rangle \subset M $ and so $b \in M$.
\end{proof}

An element $x$ of a MVW-rig $A$ is called \textbf{nilpotent} if $x^n=0$ for any $n>0$. The set $N$ of all nilpotent elements of $A$ is called the \textbf{nilradical of $A$}.

\begin{Pro}
The nilradical $N$ of a commutative MVW-rig  $A$ is an ideal of $A$ and $A/N$ doesn't have nilpotent elements different from zero.
\end{Pro}
\begin{proof}
$0 \in N$. If $x,y \in N$ then $x^n=0$ and $y^m=0$ for any $n,m \in \mathbb{N}$, then $(x \oplus y)^{m+n-1}$ is a sum of products $x^r y^s$ where $r+s= m+n-1$ and $r>n$ or $s>m$ (since $A$ is commutative), after every product is zero and therefore $(x \oplus y)^{m+n-1}=0$, so $x \oplus y \in N$. If $x \leq y \in N$ then exists $n \in \mathbb{N}$ such that $y ^n =0$, since $x \leq y$ then $x^n \leq y^n = 0$ and so $x^n =0$ therefore $x \in N$. Given $x \in N$ and $y \in A$ then there exists $n \in \mathbb{N}$ such that $x^n = 0$, since $A$ is commutative we have that $x^n y^n = (xy)^n = 0$ and therefore $xy \in N$. This shows that $N$ is an ideal. \\
To see that $A/N$ doesn't have nilpotent elements, let's take a nilpotent element $[x]_N$ in $A/N$, then there exists an integer $m>0$ such that $[x]_N^m = [0]_N$ in $A/N$, then $[x^m]_N = [0]_N$ by equation (\ref{ProdCoc}), this implies that $x^m \in N$ and therefore exists an integer $k>0$ such that $(x^m)^k =0$, then $x \in N$ and so $[x]_N = [0]_N$. 
\end{proof}

\begin{Pro}\label{NilradyPrimo}
The nilradical $N$ of a MVW-rig $A$ is contained in each prime ideal of $A$. 
\end{Pro}
\begin{proof}
Given $x \in N$, exists an integer $n>0 $ such that $x^n = 0$, since $0\in P$ for every prime ideal $P$ of $A$ then $x ^n \in P$ and since $P$ is prime, $x \in P$.
\end{proof}

\begin{Pro}
Every non-trivial MVW-rig $A$ has a maximal ideal.
\end{Pro}

\begin{proof}
Being $\Sigma $ the set of all of the proper ideals of $A$. $ \Sigma $ is different from empty since the ideal $0 \in \Sigma $, and $\Sigma $ is ordered by inclusion. Being $ (J_\alpha) $ a chain of ideals $J_1 \subset J_2 \subset J_3 \subset \cdots$ in $\Sigma$. We have that $J= \cup_\alpha J_\alpha $ is an ideal that belongs to $\Sigma$ since $u \notin J$ because $u \notin J_\alpha $ for every $\alpha$. Therefore, $J$ is an upper bound of the chain and by Zorn's lemma, $\Sigma $ has at least one maximal element.
\end{proof}

\begin{Cor}
If $I$ is an ideal of $A$ then there is a maximal ideal of $A$ that contains $I$. 
\end{Cor}
 \begin{proof}
It follows directly from the previous proposition applied to $ A / I $ and the bijection given in (\ref{Corresp_ideales}). 
\end{proof} 

\begin{Def}
Given $I$  an ideal of a commutative MVW-rig $A$, we will call \textbf{radical} of $I$ the set
\[ \sqrt{I} = \{ x \in A \mid x^n \in I \text{ for some } n > 0 \} \]
\end{Def}

\begin{Pro}
The radical of an ideal $ I $ of a commutative MVW-rig $ A $ has the following properties:
\begin{itemize}
\item[$i)$] $I \subset \sqrt{I}$.
\item[$ii)$] If $I \subset J$ then $\sqrt{I} \subset \sqrt{J}$ with $J$ ideal of $A$.
\item[$iii)$] If $I$ is the prime ideal, then $I = \sqrt{I}$.
\item[$iv)$] $\sqrt{I \cap J} = \sqrt{IJ} $
\end{itemize}
\end{Pro}

\begin{Pro}
The radical of an ideal I of a commutative MVW-rig is the intersección of the prime ideals that contain I.

\[ \sqrt{I} = \bigcap_{P \supset I} P\]
\end{Pro}
\begin{proof}
Given $x \in \sqrt{I}$ there exists $n \in N$ that $x^n \in I$, then $x^{n} \in P$ for every $P \supset I$, which implies that $x \in P$ for every $P \supset I$ because $P$ is prime ideal and therefore $x \in \bigcap\limits_{P \supset I} P$. On the other hand, given $x \notin \sqrt{I}$. Being $\Sigma$ the set of ideals $J$ that contain $I$ with the property
  
\[ n > 0 \Rightarrow x^n \notin J\]

$\Sigma$ isn't empty because $I \in \Sigma$. We must show that every chain in $\Sigma$ has a upper bound in $\Sigma$. For $J_0 \subset J_1 \subset J_2...$ the join is $ K= U_{i=0}^{\infty }J_i$. $K$ is ideal of $A$ because $0 \in K$; if $b,c \in K$ then $b\in J_{i_1}, c \in J_{i_2}$, let's suppose that $J_{i_1} \subset J_{i_2}$, then $ b \in J_{i_2}$, then $b \oplus c \in J_{i_2}$ and $b \oplus c \in K$ (the same for $J_{i_2} \supset J_{i_1}$); if $b \in K$ and $a \leq b$ then $b \in J_i$ for any $i$, then $a \in J_i$ and $a \in K$; given $b \in K$ and $a \in A$ we have that $b \in J_i$ for any $i$, then $ba \in J_i$ and so $ba \in K$. Now, by the Zorn's lemma $\Sigma$ contains at least one maximal element. Given P a maximal element of $ \Sigma $ and we'll show that it's pime. Given $z,y \notin P$ then $P \oplus \langle z \rangle $, $P \oplus \langle y \rangle $ contain strictly $P$ and therefore aren't in $\Sigma$. Then there exists integers $n,m>0$ such that

\[x^{n} \in P \oplus \langle z \rangle , \qquad x^{m} \in P \oplus\langle y \rangle \]

Then $x^{n} \leq p_1 \oplus z_1$ and $x^{m} \leq p_2 \oplus y_1$ where $p_1, p_2 \in P$, $z_1 \in \langle z \rangle $ and $y_1 \in \langle y \rangle $. We have that $x^{n+m}=x^{n}x^{m} \leq (p_1 \oplus z_1)(p_2 \oplus y_1) \leq (p_1 \oplus z_1)p_2 \oplus (p_1 \oplus z_1)y_1 \leq (p_1 \oplus z_1)p_2 \oplus p_1y_1 \oplus z_1y_1 = p_3 \oplus z_1y_1$ where $p_3 \in P$ this way we have that $x^{n+m} \in P \oplus \langle zy \rangle $. Then $P \oplus \langle zy \rangle  \notin \Sigma$ and therefore $ zy \notin P$. Then $P$ is prime. This way we have a prime ideal $P$ that contains $I$ such that $x \notin P$, then $ x \notin \bigcap\limits_{P \supset I} P$. This concludes the proof.
\end{proof}

\begin{Cor}\label{Nilradical}
The nilradical $N$ of a commutative MVW-rig $A$ is the intersection of all prime ideals of $A$.
\end{Cor}
\begin{proof}
By proposition (\ref{NilradyPrimo}) all of the prime ideals of $A$ contain the nilradical, then applying the previous proposition we arrive at the result.
\end{proof}

\section{The prime spectrum of a MVW-rig}

Now we will characterize the prime spectrum of a MVW-rig by passing prime-spectrum theorems of unitary commutative rings to said structures. Henceforth, when we speak of MVW-rig it will be understood as a commutative MVW-rig with unitary element.

\begin{Def}
Given a MVW-rig $A$, we call \textbf{prime spectrum of $A$} or $Spec(A)$ to the set of prime ideals of $A$ and for every $a \in A $ we define:
\[V(a)= \{ P \in Spec(A) \mid a \in P \} \]
\end{Def}

\begin{Pro}
The collection $ \{ V(a)\}_{a \in A} $ has the following properties that characterize a base of a topological space:
\begin{itemize}
\item[$i)$] $V(a) \cap V(b) = V(a \oplus b )$ for every $a,b \in A $
\item[$ii)$] $V(0) = Spec(A)$ 
\item[$iii)$]$V(u) = \emptyset $
\end{itemize}
\end{Pro}
\begin{proof}
$i)$ $P\in V(a) \cap V(b) \Leftrightarrow a \in P$ and $b\in P \Leftrightarrow a\oplus b \in P\Leftrightarrow P\in V(a \oplus b)$. $ii)$ $0 \in P$ for every $P \in Spec(A)$. $iii)$ $u \notin P$ for every $P \in Spec(A)$.\\
\end{proof}

From the above, we have that the collection $ \{ V(a)\}_{a \in A} $ form a base for a topology, called the \textbf{Co-Zariski topology}:

\begin{Def}
Given a MVW-rig $A$, we define the topological space $Spec(A)$ whose points are the prime ideals of $ A $ and whose open ones are generated by the base  $ \{ V(a)\}_{a \in A} $. 
\end{Def}

\begin{Pro}
Given $A$ a MVW-rig and $a,b$ elements of $A$, then:
\begin{itemize}
\item [i)] $V(a) \cup V(b) = V (ab)$
\item [ii)] $V (ab) \subset V (a \wedge b)$
\item [iii)] $V(a) \cap V(b) = V(a \vee b)$
\item [iv)] $V (a) = Spec(A)$ if and only if $ a $ is nilpotent.
\end{itemize} 
\end{Pro}

\begin{Pro}
Given $a,b$ elements of a MVW-rig $A$, then $V(a) \subset V(b)$ if and only if $\sqrt{\langle b \rangle } \subset \sqrt{\langle a \rangle }$.
\end{Pro}
\begin{proof}
Given $x \in \sqrt{\langle b \rangle }$ then $x \in \bigcap\limits_{P \supset \langle b \rangle } P$ with $P$ prime ideal, then $x \in P$ for every $P \supset \langle b \rangle $, in particular $x \in P$ for every $P \supset \langle a \rangle $ because $ V(a) \subset V(b)$, then $ x \in \bigcap\limits_{P \supset \langle a \rangle } P$ and therefore $x \in \sqrt{\langle a \rangle }$.
On the other hand, given $\sqrt{\langle b \rangle } \subset \sqrt{ \langle a \rangle }$, $ V (\sqrt{\langle a \rangle }) \subset V (\sqrt{\langle b \rangle })$ then $V(a) = V (\langle a \rangle ) = V (\sqrt{\langle a \rangle }) \subset V(\sqrt{ \langle b \rangle }) = V (\langle b \rangle ) = V(b)$.
\end{proof}
	
\begin{Pro}
Given $A$ a MVW-rig, $Spec(A)$ is a topological space $T_0$
\end{Pro}
\begin{proof}
Given $P, Q \in Spec(A)$ with $P \neq Q$, then there exists $a \in A , a \notin N$ ($N$ the nilradical of $A$), such that $a \in P$ and $a \notin Q$ or $a \notin P$ and $ a \in Q$, then $ P\in V(a)$ and $ Q \notin V (a)$ or $ P \notin V(a)$ and $ Q \in V (a)$.
\end{proof}

Let's remember that given $X$ a set of a topological space, the closure of $X$ noted by $\overline{X}$ is defined as the intersection of all closed sets containing $X$, then a point $x$ belongs to $\overline{X}$ if and only if for every basic opening $B$ that contains $x$, $ B \cap X \neq \emptyset $.

\begin{Pro}\label{clausura}
Given $Q,P \in Spec (A)$ for a MVW-rig $A$, then $Q \in \overline{ \{P \} }$ if and only if $Q \subset P$. 
\end{Pro}
\begin{proof}
If $ Q \in \overline{ \{ P \} } $ then for every $b \in Q$ we have that $V(b) \cap \{ P \} \neq \emptyset$, then $P \in V(b)$ for every $b \in Q$ and this implies that $b \in P$ for every $b \in Q$, then $Q \subset P$. On the other hand, if $ Q \subset P$, then for every $b \in Q$, $P \in V(b)$, which implies that $ V(b) \cap \{ P \}  \neq \emptyset $ for every $b \in Q$ and therefore $Q \in \overline{ \{ P \} }$.
\end{proof}

\begin{Pro}\label{clausura1}
Given $Q \in Spec (A)$ for a MVW-rig $A$, and $U$ a subset of $Spec(A)$. If $Q \subset P$ for any $P \in U$ then $Q \in \overline{U} $
\end{Pro}

\begin{Obs}\label{clausura2}
The opposite of the previous proposition is true if the set $ U $ has a single maximal element.
\end{Obs}

A topological space $X$ is irreducible if $X \neq \emptyset $ and the intersection of two non-empty openings is non-empty.

\begin{Pro}
For a MVW-rig $A$, $Spec(A)$ is irreducible if and only if $A$ has a single maximal ideal.
\end{Pro}
\begin{proof}
Let's suppose that $A$ has at least two maximal ideals $M_1$ and $M_2$, then given $a \in M_1, a \notin M_2$ there exists $b \in M_2$ such that $ x \oplus b = 1$ with $x \in \langle a \rangle $, and $b \notin M_1$ because $M_1$ is it's own ideal. It results that $V(x)$ and $V(b)$ are non-empty openings because $M_1 \in V(x)$ and $ M_2 \in V(b)$, then $ \emptyset = V(\langle 1 \rangle ) = V( \langle x \oplus b \rangle ) = V(x \oplus b ) = V(x) \cap V(b)$, that is to say,  $V(x) \cap V(b) = \emptyset $ which implies that $Spec(A)$ is not irreducible. On the other hand, if $A$ has exactly a maximal ideal, let's take $V(a) \neq \emptyset$ and $V(b) \neq \emptyset$, this implies that $a \in M$ and $b \in M$, then $M \in V(a) \cap V(b)$ and therefore $A$ is irreducible. 
\end{proof}

\begin{Pro}\label{funcionSpec}
Given $\phi : A \rightarrow B$ a homomorphism of MVW-rigs, we define $\phi ^{*} : Spec (B) \rightarrow Spec (A)$ such that given $J$ ideal of $B$, $\phi ^{*}(J) = \begin{Bmatrix} 	x \in A \mid \phi (x) \in J	\end{Bmatrix}= \phi^{-1}(J)$. Then:
\begin{itemize}
\item [i)] $\phi ^{*}$ is a continuous function between topological spaces.
\item [ii)] If $I$ is an ideal of A then $(\phi^{*}) ^{-1} (V(I)) = V( \phi(I))$
\item [iii)] If $\phi$ is injective, $\phi^* (V(b)) = V ( \phi^{-1} (b))$
\item [iv)] If $\phi$ is bijective, then $\phi^*$ is a homeomorphism between $Spec(B)$ and $V(Ker(\phi))$.
\item [v)] If $\phi$ is injective, then $\phi^*(Spect(B)) = Spec(A)$. 
\end{itemize}
\end{Pro}
\begin{proof}
\begin{itemize}
\item [i)] First, let's see that for $J \in Spec (B), \phi ^{*} (J) \in Spec (A)$. $ 0 \in \phi ^{*} (J)$ because $\phi (0) = 0 \in J$. Given $ x, y  \in \phi ^{*} (J)$ then $ \phi (x), \phi (y) \in J$, then $\phi (x) \oplus \phi (y) = \phi ( x\oplus y) \in J$ and so $ x \oplus y \in \phi ^{*} (J)$. Given $x \in \phi ^{*} (J)$ and $y \leq x$, then $ \phi (y) \leq \phi (x)$ and as $\phi (x) \in J$ then $\phi (y) \in J$, so $y \in \phi ^{*}(J)$. Given $x \in \phi ^{*} (J)$ and $a \in A$, then $\phi (x) \in J$, as $J$ is ideal then $\phi (a) \phi (x) = \phi (ax) \in J$, then $ax \in \phi ^{*} (J)$. To prove it's prime let's take $ xy \in \phi ^{*} (J)$ then $ \phi (xy) \in J$ and as $J$ is prime then $\phi (x) \in J$ or $ \phi (y) \in J$ then $x \in \phi ^{*} (J)$ or $ y \in \phi ^{*} (J)$. \\
Now let's prove that $\phi^{*}$ is continuous. Given $V(a) \in Spec(A)$ let's demonstrate that $(\phi^{*}) ^{-1} (V(a))$ is an opnening in $Spec(B)$. $(\phi^{*})^{-1} (V(a)) =  \{P \in Spec(B) \mid \phi^{*} (P) \in V(a)\} = \{ P \in Spec(B) \mid a \in \phi^{*} (P) \} = \{ P \in Spec(B) \mid \phi(a) \in P\} = V(\phi(a)) \in Spec(B)$.

\item [ii)] Given $P \in Spec(A)$, $P \in (\phi^{*}) ^{-1} (V(I)) \Leftrightarrow \phi^{*} (P) \in V (I) \Leftrightarrow I \subset \phi^{*} (P) \Leftrightarrow \phi (I) \subset P \Leftrightarrow  P  \in V(\phi (I)) $.

\item [iii)] $\phi^*(V(b))= \{ Q \in Spec(A) | Q = \phi^{-1}(P), P \in V(b) \} = \{ Q \in Spec(A) | Q = \phi^{-1}(P), b \in P \} = \{ Q \in Spec(A) | b \in \phi(Q) \} = \{ Q \in Spec(A) | \phi^{-1}(b) \in Q \} = V(\phi^{-1}(b))$

\item [iv)] If $Q \in Spec(B)$ then $Ker(\phi)$ is contained in $\phi^*(Q)$. If $P \in V(Ker(\phi))$ then $P/  Ker(\phi)$ is isomorphic with an ideal $Q$ of $Spec(B)$ under the isomorphism $\bar{\phi}: A/  Ker(\phi) \mapsto B$ given that $\phi $ is surjective. So, $P = \phi^*(Q)$ and $\phi^*$ are surjective over $V(Ker(\phi))$. Now, if $\phi^*(P) = \phi^*(Q)$ then $\phi^{-1}(P) = \phi^{-1}(Q)$ and as $\phi$ is bijective, then $P = Q$, which shows that $\phi^* $ is injective. We had already shown the continuity of $\phi^*$ in ($i$), it only remains to show that $\phi^{-1}$ is continuous, that is to say, that $\phi^*$ is an open function, but this is obtained from ($iii$) by $\phi$ being injective. 

\item [v)] Note that $\phi$ is injective, by ($iii$) we have that $\phi^*(Spec(B)) = \phi^*(V(0))= V(\phi^{-1}(0)) = V(Ker(\phi)) = V(0) = Spec(A)$. 
\end{itemize}
\end{proof}

\section{Compactness of the prime spectrum of a MVW-rig}

In this section we will show that the prime spectrum of a MVW-rig $ A $ is compact, for this the filters will be used.

\begin{Def}[Filtro]
 Given $A$ a MVW-rig, a non-empty subset $F$ of $A$ is \textbf{filter} of $A$ if fulfills the following conditions for every $a,b \in A$:
 \begin{itemize}
 \item[\textbf{F1)}] If $a \leq b $ and $a \in F$ then $b \in F$,
 \item[\textbf{F2)}] If $a, b \in F$ then $ab \in F$
\end{itemize}  
\end{Def} 

We define the \textbf{P-filters} as the filters $F$ that fulfill an aditional property: 
\begin{itemize}
\item[\textit{\textbf{F3)}}] Given $x \in A$ and $\bigoplus_i b_i x \in F$ a finite addition with $b_i \in A$ for every $i$, then $x \in F$.
\end{itemize}

\begin{Pro}\label{pfilter}
For a MVW-rig $A$ and a set $S \subseteq A$, the P-filter generated by $S$ in $A$ is,
\[ \langle S \rangle_P = \{ x \in A \mid (\exists s_1, \dots , s_n \in S), (\exists b_1, \dots , b_m \in A): s_1 \cdots s_n \leq \bigoplus_i b_i x \} \]
 
\end{Pro}
\begin{proof}
It is easy to see that $S \subset \langle S \rangle_P$ because $s^2 \leq s^2$ implies that $s \in \langle S \rangle_P$. Let's see that $\langle S \rangle_P $ is P-filter: $i)$ Given $x \leq y, x\in \langle S \rangle_P$, then there exist $s_1, \dots , s_n \in S$ and $b_1, \dots , b_m \in A$ such that $  s_1 \cdots  s_n \leq \bigoplus_i b_i x$, and as the operations conserve the order, then $\bigoplus_i b_i x \leq \bigoplus_i b_i y$ and therefore $ s_1 \cdots  s_n \leq \bigoplus_i b_i y$, that is to say, $y \in \langle S \rangle_P$. $ii)$ Given $x,y \in \langle S \rangle_P $, then there exist $s_1, \dots , s_{n_1}, t_1, \dots , t_{n_2} \in S$ and $b_1, \dots , b_{m_1}, c_1, \dots, c_{m_2} \in A$ such that $ s_1 \cdots s_{n_1} \leq \bigoplus_i b_i x$ and $ t_1 \cdots  t_{n_2} \leq \bigoplus_j c_j y $, then $s_1 \cdots s_{n_1} \cdot t_1 \cdots  t_{n_2}\leq (\bigoplus_i b_i x)(\bigoplus_j c_j y) \leq \bigoplus_k d_k xy $ where $d_k$ are products $b_ic_j$, so $xy \in \langle S \rangle_P$. Now let's demonstrate that $\langle S \rangle_P$ has the P-filter property: If $\bigoplus_j c_j x \in \langle S \rangle_P$ for some $c_j \in A$ then there exist $s_1, \dots , s_n \in S$ and $b_1, \dots , b_m \in A$ such that $  s_1 \cdots  s_n \leq \bigoplus_i b_i (\bigoplus_j c_j x)\leq  \bigoplus_{ij} b_i c_j x = \bigoplus_k d_k x$ and therefore $x \in \langle S \rangle_P$. \\
Finally, let's see that it is the smallest P-filter containing $S$. Given $H$ P-filter such that $S \subset H$, we want to see that $\langle S \rangle_P \subset H$. Given $x \in \langle S \rangle_P$, there exist $b_1, \dots , b_m \in A$ such that $ s_1 \cdots  s_n \leq \bigoplus_i b_i x$ with  $s_1, \dots , s_n \in S \subset H$, then $s_1 \cdots  s_n \in H$ by being $H$ a filter, then $\bigoplus_i b_i x \in H$ and as $H$ has the P-filter property, we have that $x \in H$.
\end{proof}
Then $\langle S \rangle_P  $ is the smallest P-filter that contains $S$.
In particular, for an element $a$ of $A$ the P-filter generated by $a$ is:
\begin{equation}
F_a = \{  x \in A \mid  \exists n \in \mathbb{N} \text{ and } b_1, \dots , b_m \in A \text{ such that }  a^n \leq \bigoplus_i b_i x \}
\end{equation}
is a P-filter.

\begin{Pro} 
Every P-filter F of $A$ satisfies that:
\begin{center}
$F = \bigcup\limits_{a \in F} F_{a}$
\end{center}
\end{Pro}
\begin{proof}
Given $x \in F$, then $x \in F_{x}\subseteq \bigcup\limits_{a \in F} F_{a}$. On the other hand, given $y \in \bigcup\limits_{a \in F} F_{a}$, then $y \in F_a$ for some $a \in F$, there exist $n \in \mathbb{N}$ and $b_1, \dots , b_m \in A$ such that $a^n \leq \bigoplus_i b_i y$, as $a \in F$, then $a^n \in F$ and therefore $\bigoplus_i b_i y \in F$, since $F$ has the P-filter property, $y \in F$.
\end{proof}

The union of two P-filters is not necessarily a P-filter. The following proposition defines the join and the meet of P-filters, which are P-filters.

\begin{Teo}\label{filtroLocal}
Given $A$ a MVW-rig we have that: \\
\begin{itemize}
	\item[i)] The P-filter generated by the union of two P-filters $F_a$ and $F_b$ is:
	\[ \langle F_a \cup F_b \rangle_P = \{ x \in A \mid \exists n_1,n_2 \in \mathbb{N} \text{ and } b_1, \dots , b_m \in A \mid a^{n_1} b^{n_2} \leq \bigoplus_i b_i x \} \] 
	\item[ii)] $ \bigvee_{a \in I} F_a = \langle \bigcup_{a \in I} F_a \rangle_P $
	\item[iii)] $F_a \cap F_b  = F_{a \vee b}$,
	\item[iv)] $F_a \cap F_b = F_a \wedge F_b $
	\item[v)] $ \langle F_a \cup F_b \rangle_P = F_a \vee F_b = F_{ab} $,
\end{itemize} 
\end{Teo}
\begin{proof}
$i)$ and $ii)$ are followed directly from the proposition (\ref{pfilter}). \\
$iii)$ Given $x \in F_a \cap F_b$ then there exist $n_1,n_2 \in \mathbb{N}$ y $b_1, \dots , b_{m_1}, c_1, \dots , c_{m_2} \in A$ such that $a^{n_1} \leq \bigoplus_i b_i x$ and $b^{n_2} \leq \bigoplus_j c_j x$, then $(a \vee b)^{n_1 + n_2} \leq \bigoplus_k d_k x \oplus \bigoplus_l e_l x$ because $(a \vee b)^{n_1 + n_2} \leq (a \oplus b)^{n_1 + n_2}$ and when expanding we get terms of the form $a^s b^r$ where $s >n_1$ or $r>n_2$, and therefore, without loss of generality, for $s > n_1$,  $a^s b^r \leq \bigoplus_i b_i x b^r \leq  \bigoplus_k d_k x$, and by expanding the expression arriving at what we wanted to prove. In this way $x \in F_{a \vee b}$. On the other hand, given $x \in F_{a \vee b}$, then there exist $n \in \mathbb{N}$ and $b_1, \dots , b_m \in A$ such that $(a \vee b)^n \leq \bigoplus_i b_i x$ and by property (v) of the proposition (\ref{propiedades}) $a^n \vee b^n \leq \bigoplus_i b_i x$, then $a^n \leq \bigoplus_i b_i x$ and $b^n \leq \bigoplus_i b_i x$ and so $x \in F_a \cap F_b$. \\
$iv)$ It follows directly from the fact that $F_a \cap F_b $ is a P-filter, like we showed it before. \\ 
$v)$ Given $x \in \langle F_a \cup F_b \rangle_P$ then there exist $n_1,n_2 \in \mathbb{N}$ and $b_1, \dots , b_m \in A$ such that $a^{n_1} b^{n_2} \leq \bigoplus_i b_i x$; if $n_1 \leq n_2$ then $a^{n_2} b^{n_2} \leq \bigoplus_i c_i x$ where $c_i = a^{n_2-n_1} b_i$, then  $(ab)^{n_2} \leq \bigoplus_i c_i x$ and therefore $x \in F_{ab}$. On the other hand, given $x \in F_{ab}$ then there exists $n \in \mathbb{N}$ and $b_1, \dots , b_m \in A$ such that $(ab)^n \leq \bigoplus_i b_i x$, this is, $a^n b^n \leq \bigoplus_i b_i x$ and so $x \in \langle F_a \cup F_b \rangle_P$. \\ 
\end{proof}

\begin{Teo}\label{L_A}
Given a MVW-rig $A$, the collection $L_A$ of P-filters of $A$, is a local.
\end{Teo}
\begin{proof}
We want to see that  $ F\wedge \bigvee\limits_{a \in I} F_{a}=\bigvee\limits_{a\in I} ( F \wedge F_{a})$. \\

First, note that $F \cap \langle \bigcup\limits_{a \in I}F_{a} \rangle_P = \langle F \cap \bigcup\limits_{a \in I}F_{a} \rangle_P$, in fact, since $F \cap \bigcup\limits_{a \in I}F_{a} \subset F,\bigcup\limits_{a \in I}F_{a}$  then $ \langle F \cap \bigcup\limits_{a \in I}F_{a} \rangle_P \subset \langle F \rangle_P = F$ and $  \langle F \cap \bigcup\limits_{a \in I}F_{a} \rangle_P \subset \langle \bigcup\limits_{a \in I}F_{a} \rangle_P$, then $  \langle F \cap \bigcup\limits_{a \in I}F_{a} \rangle_P \subset F \wedge \langle\bigcup\limits_{a \in I}F_{a} \rangle_P = F \cap \langle \bigcup\limits_{a \in I}F_{a} \rangle_P $.

On the other hand, given $x \in  F \cap \langle \bigcup\limits_{a \in I}F_{a} \rangle_P$ then $x \in F $ and $x \in \langle \bigcup\limits_{a \in I}F_{a} \rangle_P$,  therefore there exist finite $ s_i\in F_{a_i}, b_1 \dots b_m \in A$ such that $ s_1 \cdots s_k \leq \bigoplus_j b_j x$. Note that $(x \oplus s_1) \cdots (x \oplus s_k) \leq x^k \oplus x^{k-1}s_1 \oplus \cdots \oplus    s_1 \cdots s_k \leq \bigoplus_l c_l x$. By axiom \textbf{F1} of filter, we have that $x \oplus s_i \in F \cap F_{a_i} $ for each $ i= 1, \dots ,k $ therefore $x \in \langle F \cap \bigcup\limits_{a \in I}F_{a} \rangle_P$

Now, if we use subsection (ii) of the proposition (\ref{filtroLocal}) we have:
$ F\wedge \bigvee\limits_{a \in I} F_{a} = F \wedge  \langle \bigcup\limits_{a \in I}F_{a} \rangle_P =  F \cap  \langle \bigcup\limits_{a \in I}F_{a} \rangle_P = \langle F \cap \bigcup\limits_{a \in I} F_a \rangle_P = \langle \bigcup\limits_{a \in I} ( F \cap F_a)  \rangle_P = \bigvee\limits_{a\in I} ( F \wedge F_{a})$.
\end{proof}

\begin{Teo}
Given $A$ a MVW-rig, then $L_A$ is compact.
\end{Teo}
\begin{proof}
Given $\bigvee\limits_{a \in I} F_{a}= A$ we want to see that there is a finite subcollection of $\left\{F_{a}\right\}_{a \in J}$ such that $\bigvee\limits_{a \in J} F_{a}= A$, $J$ finite set, $J\subset I$. \\
It results that $\bigvee\limits_{a \in I} F_{a} = \langle \bigcup\limits_{a \in I}F_{a} \rangle_P = \{ x\in A \mid \exists n_j \in \mathbb{N},$ with $ 0 \leq j \leq k$ and $b_1, \dots , b_m \in A$ such that $  a_1^{n_1} a_2^{n_2} \cdots a_k^{n_k} \leq \bigoplus_i b_i x \} $. \\
Since $0 \in A$, then $0 \geq a_1^{n_1} a_2^{n_2} \cdots a_k^{n_k}$, with $n_j \in \mathbb{N}$,  then $0 \in \langle \bigcup\limits_{j=1}^{m}F_{a_{j}} \rangle_P =  \bigvee\limits_{j=1}^{m}F_{a_{j}}$, so $\bigvee\limits_{j=1}^{m}F_{a_{j}}=A$.
\end{proof}

\begin{Obs}
	$\mathcal{O}(Spec(A))$ is the set of open subsets of $Spec(A)$
\end{Obs}

\begin{Teo}\label{local} 
For a MVW-rig $A$, the function of locals $\theta$ between $Spec(A)$ and $L_A.$
\begin{center}
$\mathcal{O}(Spec(A))\stackrel{\theta}{\rightarrow}L_{A}$ 

$V(a) \longrightarrow F_{a}$

$\bigcup\limits_{j\in J} V(a_j) \longrightarrow \bigvee\limits_{j\in J}F_{a_{j}} $
\end{center}
Is an isomorphism.
\end{Teo}
\begin{proof}
First, let's see that $\theta$ is an isomorphism: 
\begin{itemize}
\item $\theta(V(a) \cup V(b)) = \theta(V(ab)) = F_{ab} = F_a \vee F_b = \theta(V(a)) \vee \theta(V(b))$. 
\item $\theta(V(a) \cap V(b)) = \theta(V(a \vee b)) = F_{a \vee b} = F_a \wedge F_b = \theta(V(a)) \wedge \theta(V(b))$.
\end{itemize}
$\theta$ is surjective: given that $F \in L_A$, we have that $F= \langle \bigcup\limits_{a\in F} F_{a} \rangle_P = \bigvee\limits_{a\in F} F_{a}$ then $\theta\textit{}(\bigcup\limits_{a\in F}  V(a))=F$ by definition of $\theta$. \\
Now, let's see that $\theta$ is injective: given  $F_{1}, F_{2}\in L_A$, such that $F_{1} \neq F_{2}$ being $x \notin F_{2}$ and $x \in F_{1}$ then the ideal of the MVW-rig $A$ generated by $x$ satisfies that $\langle x \rangle \cap F_2 = \emptyset$, in fact, $\langle x \rangle \cap F_{2} \neq \emptyset$, implies that there exists  $z \in (x) \cap F_{2}$, and $z \leq \bigoplus_j b_i x$, since $F_{2}$ is P-filter,  $x \in  F_{2}$, which is absurd.

Let's consider the set of ideals,

\begin{center}
$\Sigma =\{I \mid x \in I; I\cap F_{2} = \emptyset \}$

$(x) \in \Sigma$  then  $ \Sigma \neq \emptyset$
\end{center}

The set  $\Sigma$ is inductively superior: each chain of ideals $ I_{i}\in \Sigma$ has a upper bound $\bigcup I_{i}\in \Sigma$. Then by the Zorn's lemma, $\Sigma$ contains at least a maximal element. Being $P$ a maximal of $\Sigma$,then  $x \in P$ and $P\cap F_2=\emptyset$. We want to see that P is prime ideal: being $y, z \in A$, such that $yz \in P$; we want to see that  $y \in P$ or $z \in P$. Let's suppose that $y \notin P$ and $z \notin P$. By the maximality of $P$ in $\Sigma$, we follow that: 

\begin{center}
$ \langle P \cup \left\{y\right\} \rangle \cap F_{2} \neq \emptyset$\qquad{}  and \qquad{} $ \langle P \cup \left\{z\right\} \rangle \cap F_{2} \neq \emptyset$ 
\end{center}

Consequently there are $p,q \in P$, $w ,w' \in F_{2}$ and $a_1, \dots, a_{m_1}, b_1, \dots , b_{m_2} \in A $ such that

\begin{center}
$w \leq p \oplus \bigoplus_i a_i y$\qquad{}  and \qquad{} $w' \leq q \oplus \bigoplus_j b_j z $
\end{center}

Since  $F_{2}$ is filter, then 

\begin{center}
$(p \oplus \bigoplus_i a_i y)\in F_{2}$\qquad{}  and \qquad{} $(q \oplus \bigoplus_j b_j z)\in F_{2}$
\end{center}

So, as the product retains order, we have: 

\begin{center}
$ w w' = (p \oplus \bigoplus_i a_i y)(q \oplus \bigoplus_j b_j z) \leq r \oplus \bigoplus c_i yz = r'$
\end{center}

where the axiom (iii) was used and $r$ is an element of $P$ obtained from the sums and products of the elements $p,q \in P$ with other elemets of $A$ that, by absorbent property of $P$ are in $P$. Since $w w' \in F_2$ and  $F_2$ is filter $r' \in F_{2}$. We follow that $r' \in P\cap F_{2}$ which contradicts the hypothesis  $P\cap F_{2}=\emptyset $; then $y \in P$ or $z \in P$. In consequence

\begin{center}
$\bigcup\limits_{a\in F_{1}} V(a) \neq  \bigcup\limits_{b\in F_{2}} V(b) $.  
\end{center}

because $P \in \bigcup\limits_{a\in F_{1}} V(a) $, due to $ x \in F_{1}$, $P \in V(x)$, $P \notin \bigcup\limits_{b \in F_{2}}V(b)$  due to  $P\cap F_{2}=\emptyset$
\end{proof}

\begin{Teo} 
Given $A$ a MVW-rig, the $Spec(A)$ with the co-Zariski topology, is a compact topological space.
\end{Teo}
\begin{proof}
The preceding statements are followed.
\end{proof}
\section{Conclusions}
The MV-algebras were founded by Chang \cite{Chang2} to demostrate a completeness theorem for fuzzy logic. This rich structure has been studied since many years ago. One of the most important difficulties for someone to make commutative fuzzy algebra is the absense of some product operation. The MV-algebra $[0,1]$ has a natural product and this product respects the MV-algebra structure. There are some results about the MV-algebra's product \cite{DiNola} and \cite{DiNola2}. Our focus is related with optaning an adequate theory in order to represent a fuzzy commutative algebra in the best form.

There exists a close relationship between the class of the special MVW-rigs and some kind of $l_u-$rings. These categories are equivalent, however we don't show this result here.
%
%

\end{document}